\documentclass[12pt]{amsart}%
\usepackage{amsmath, amssymb,tabularx,graphicx}
\usepackage{cite}
\usepackage[mathscr]{eucal}
\usepackage{amsmath}
\usepackage{amsfonts}
\usepackage{amssymb}
\usepackage{graphicx}%
\newtheorem{thm}{Theorem}

\newtheorem{lemma}[thm]{Lemma}
\newtheorem{defn}[thm]{Definition}

\newtheorem{cor}[thm]{Corollary}
\newtheorem{remark}{Remark}
\newtheorem{ex}[thm]{Example}

\newcommand{\Z}{\mathscr{Z}}

\begin{document}
\title[A Comparison of Large Scale Dimension]{A Comparison of Large Scale Dimension of a Metric Space to the Dimension of its Boundary}
\author{Craig R. Guilbault and Molly A. Moran}
\thanks{Work on this project was aided by a Simons Foundation Collaboration Grant.}

\begin{abstract}
Buyalo and Lebedeva have shown that the asymptotic dimension of a hyperbolic group is equal to the dimension of the group boundary plus one. Among the work presented here is a partial extension of that result to all groups admitting $\Z$-structures; in particular, we show that $\hbox{asdim}G\geq \hbox{dim}Z+1$ where $Z$ is the $\Z$-boundary. 
\end{abstract}
\maketitle

\section{Introduction}

The primary goal of this paper is to establish a connection between the
asymptotic dimension of a group admitting a $\mathscr{Z}$-structure and the
covering dimension of the group's boundary.

For hyperbolic $G$, the relationship is strong; Buyalo and Lebedeva
\cite{BuLe07} have shown that $\hbox{asdim}G=\hbox{dim}\partial G+1$. In
\cite{ChHo12}, a partial extension to CAT(0) groups was attempted.
Specifically, it was claimed that $\hbox{asdim}G\geq\hbox{dim}\partial G+1$,
where $\partial G$ is any CAT(0) boundary of $G$. However, in MathSciNet
review MR3058238, X. Xie pointed out a critical error in the proof. Here we
recover the same inequality as a special case of a more general theorem.

\begin{thm}
\label{Theorem: dimZ+1<=asdimG} Suppose a group $G$ admits a $\mathscr{Z}$%
-structure, $(\overline{X},Z)$. Then $\dim Z+1\leq\operatorname*{asdim}G$.
\end{thm}

By a $\mathscr{Z}$-structure on $G$, we are referring to the axiomatized
approach to group boundaries laid out in \cite{Be96} and expanded upon in
\cite{Dr06}. Groups known to admit $\mathscr{Z}$-structures include:
hyperbolic groups (with $X$ being a Rips complex and $Z=\partial G$)
\cite{BeMe91}; CAT(0) groups (with $X$ being the CAT(0) space and $Z$ its
visual boundary) \cite{Be96}; systolic groups \cite{OsPr09}, Baumslag-Solitar
groups \cite{GuMoTi15}; as well as various combinations of these classes, as described in
\cite{Tir11}, \cite{Da03}, and \cite{Ma14}. Definitions of $\mathscr{Z}$-structure and other
key terms used here will be provided in the next section.

Theorem \ref{Theorem: dimZ+1<=asdimG} will be obtained from a more general
observation about metric spaces.

\begin{thm}
\label{Theorem: controlled Z-compactification} Suppose a proper metric space
$(X,d)$ admits a controlled $\mathscr{Z}$-compactification $\overline{X}=X\cup
Z$. Then $\dim Z+1\leq\dim_{\emph{\text{mc}}}X$.
\end{thm}

Here, $\dim_{\text{mc}}$ stands for Gromov's \emph{macroscopic dimension}, a
type of large scale dimension for metric spaces that is less restrictive than
asymptotic dimension in that, for any $\left(  X,d\right)  $,
$\operatorname*{dim}_{\text{mc}}X\leq\hbox{asdim}X$. To complete the proof of
Theorem \ref{Theorem: dimZ+1<=asdimG} it will then suffice to show that, for a
$\mathscr{Z}$-structure $(\overline{X},Z)$ on a group $G$, $\overline{X}$ is a
controlled $\mathscr{Z}$-compactification and $\hbox{asdim}X=\hbox{asdim}G$.

Theorem \ref{Theorem: controlled Z-compactification} is inspired by the main
argument in \cite{GuTi13} together with the point of view presented in \cite{Mo14}.

\section{Background and Definitions}

We begin by providing a few definitions and results for the different
dimension theories and then we discuss controlled $\mathscr{Z}$%
-compactifications and $\mathscr{Z}$-structures.

Given a cover $\mathscr{U}$ of a metric space $X$, mesh$(\mathscr{U})=\sup
\{\text{diam}(U)|U\in\mathscr{U}\}$. The cover is \textbf{uniformly bounded}
if there exists some $D>0$ such that mesh$(\mathscr{U})\leq D$. The
\textbf{order} of $\mathscr{U}$ is the smallest integer $n$ for which each
element $x\in X$ is contained in at most $n$ elements of $\mathscr{U}$.

\begin{defn}
The \textbf{covering dimension} of a space $X$ is the minimal integer $n$ such
that every open cover of $X$ has an open refinement of order at most $n+1$.
\end{defn}

There are various ways to show that a space has finite covering dimension.
When working with compact metric spaces, we prefer the following.

\begin{lemma}
For a compact metric space $X$, $\dim X\leq n$ if and only if, for every
$\epsilon>0$, there is an open cover $\mathscr{U}$ of $X$ with \emph{mesh}%
$\mathscr{U}<\epsilon$ and \emph{order}$\mathscr{U}\leq n+1$.
\end{lemma}

Covering dimension can be thought of as a small-scale property. Gromov
introduced asymptotic dimension as a large scale analog of covering dimension
\cite{Gr93}.

\begin{defn}
The \textbf{asymptotic dimension} of a metric space $X$ is the minimal integer
$n$ such that for every uniformly bounded open cover $\mathscr{V}$ of $X$,
there is a uniformly bounded open cover $\mathscr{U}$ of $X$ with
\emph{order}$(\mathscr{U})\leq n+1$ so that $\mathscr{V}$ refines
$\mathscr{U}$. In this case, we write $\operatorname*{asdim}X=n$.
\end{defn}

For a nice survey of asymptotic dimension, see \cite{BDr07}. Although Theorem
\ref{Theorem: dimZ+1<=asdimG} is stated for asymptotic dimension, we will
prove a stronger result using a weaker notion of large scale dimension known
as \emph{(Gromov) macroscopic dimension}.

\begin{defn}
The \textbf{Gromov macroscopic dimension} of a metric space $X$ is the minimal
integer $n$ such that there exists a uniformly bounded open cover of $X$ with
order at most $n+1$. In this case, we write $\dim_{\emph{\text{mc}}}X=n$.
\end{defn}

Clearly $\dim_{\text{mc}}X\leq\operatorname*{asdim}X$ for every metric space
$X$.

As noted in the introduction, Theorem \ref{Theorem: dimZ+1<=asdimG} about
groups and their boundaries will be deduced from a broader observation about
certain $\mathscr{Z}$-compactifications of metric spaces. Recall that a closed
subset, $A$, of an ANR, $Y$, is a \textbf{\emph{\boldmath$\mathscr{Z}$-set}}
if there exists a homotopy $H:Y\times\lbrack0,1]\rightarrow X$ such that
$H_{0}=\operatorname*{id}_{Y}$ and $H_{t}(X)\subset Y-A$ for every $t>0$.

\begin{defn}
A \textbf{controlled \boldmath$\mathscr{Z}$-compactification} of a proper
metric space $X$ is a compactification $\overline{X}=X\cup Z$ satisfying the
following two conditions:

\begin{itemize}
\item $Z$ is a $\mathscr{Z}$-set in $\overline{X}$

\item For every $\epsilon>0$ and every $R>0$, there exists a compact set
$K\subset X$ so that every ball of radius $R$ in $X$ not intersecting $K$ has
diameter less than $\epsilon$ in $\overline{X}$.
\end{itemize}

In this case, $Z$ is called a \textbf{\boldmath$\mathscr{Z}$-boundary}, or
simply a \textbf{boundary} for $X$.
\end{defn}

There are a few things to take note of in the above definition. First, we have
followed tradition and defined $\mathscr{Z}$-sets in ANRs; hence the
compactification $\overline{X}$ must be an ANR. Furthermore, since open
subsets of ANRs are also ANRs, $X$ must be an ANR to be a candidate for a
controlled $\mathscr{Z}$-compactification\footnote{See Remark
\ref{Remark: generalized Z-sets}.}. Secondly, it is important to distinguish
between the (proper) metric $d$ on $X$ and the metric $\overline{d}$ on
$\overline{X}$. The second condition, which we call the \emph{control
condition}, says balls of radius $R$ in $(X,d)$ get arbitrarily small near the
boundary, when viewed in $(\overline{X},\overline{d})$. The metric $d$ is
crucial; it is given in advance and determines the geometry of $X$. For our
purposes the metric on $\overline{X}$ is arbitrary; any $\overline{d}$
determining the appropriate topology can be used.

\begin{ex}
The addition of the visual boundary to a proper CAT(0) space is a prototypical
example in Geometric Group Theory of a controlled $\mathscr{Z}$-compactification.
\end{ex}

In the presence of nice group actions, controlled $Z$-compactifications arise
rather naturally. As a result, our discussion can be extended to asymptotic
dimension of groups and covering dimension of group boundaries. The following
definition is key.

\begin{defn}
A \textbf{\boldmath$\mathscr{Z}$-structure on a group \boldmath$G$} is a pair
of spaces $(\overline{X},Z)$ satisfying the following four conditions:

\begin{enumerate}
\item $\overline{X}$ is a compact AR,

\item $Z$ is a $\mathscr{Z}$-set in $\overline{X}$,

\item $X=\overline{X}-Z$ is a proper metric space on which $G$ acts properly,
cocompactly, by isometries, and

\item $\overline{X}$ satisfies a \emph{nullity condition} with respect to the
action of $G$ on $X$: for every compact $C\subseteq X$ and any open cover
$\mathscr{U}$ of $\overline{X}$, all but finitely many $G$ translates of $C$
lie in an element of $\mathscr{U}$.
\end{enumerate}
\end{defn}

\begin{remark}
\emph{This definition of} $\mathscr{Z}$\emph{-structure is due to
Dranishnikov \cite{Dr06}. It generalizes Bestvina's original definition from
\cite{Be96} by allowing }$\overline{X}$\emph{ to be infinite-dimensional and
}$G$\emph{ to have torsion. We have added an explicit requirement that the
metric on }$X$\emph{ be proper; a quick review of \cite{Dr06} reveals that
this requirement was assumed there as well.}
\end{remark}

\section{Proofs}

We begin with a proof of Theorem \ref{Theorem: controlled Z-compactification},
as the other results will be obtained from it. A key ingredient is the
following classical fact about covering dimension.

\begin{lemma}
\cite{Hu35} For any nonempty locally compact metric \ space $X$, $\dim
(X\times\lbrack0,1])=$\newline$\dim X+1$.
\end{lemma}

\begin{proof}
[Proof of Theorem \ref{Theorem: controlled Z-compactification} ]Suppose $X$
admits a controlled $\mathscr{Z}$-compactification, $\overline{X}= X\cup Z$,
and let $\epsilon>0$. Assume that $\operatorname*{dim_{mc}}X=n$ and let
$\mathscr{U}$ of $X$ be a uniformly bounded open cover with
order$(\mathscr{U})\leq n +1$.

Using the control condition, we may choose a compact set $K_{0}$ such that
diam$_{\overline{d}}U\leq\frac{\epsilon}{3}$ for every $U\in\mathscr{U}$ with
the property that $U\cap K_{0}=\emptyset$. Let $\mathscr{U}^{\prime}%
=\{U\in\mathscr{U}|U\cap K_{0}=\emptyset\}$.

Since $Z$ is a $\mathscr{Z}$-set, there is a homotopy $J:\overline{X}%
\times\lbrack0,1]\rightarrow\overline{X}$ such that $J_{0}=\operatorname*{id}%
{}_{\overline{X}}$ and $J_{t}(\overline{X})\cap Z=\emptyset$ for all $t>0$. By
compactness there is some $T>0$ such that $\overline{d}(z,J_{t}(z))<\frac
{\epsilon}{3}$ for all $z\in Z$ and $t\in\lbrack0,T]$. Furthermore, we may
choose $T^{\prime}>0$ so that $J(Z\times(0,T^{\prime}])\subset\bigcup
_{U\in\mathscr{U}^{\prime}}U$. Set $t_{0}=\hbox{min}\{T,T^{\prime}\}$.

Define $H:\overline{X}\times\lbrack0,1]\rightarrow\overline{X}$ by setting
$H(x,t)=J(x,t_{0}\cdot t)$. Restrict $H$ to $Z\times\lbrack0,1]$. We will
reparametrize $H:Z\times\lbrack0,1]\rightarrow\overline{X}$ in a manner
similar to \cite{GuTi13}, so that pre-images of the open sets in
$\mathscr{U}^{\prime}$ have small mesh. After one additional adjustment, those
pre-images will form the desired cover of $Z\times\lbrack0,1]$. For
convenience we will use the $\ell_{\infty}$ metric on $Z\times\lbrack0,1]$,
$d_{\infty}=\hbox{max}\{\overline{d},|\cdot|\}$, where $|\cdot|$ is the
standard metric on $[0,1]$.

Pick $n\in\mathbb{Z}^{+}$ so that $\frac{3}{n}<\frac{\epsilon}{3}$. Choose
$t_{1}>t_{2}>\cdots>t_{n+1}\in\left[  0,1\right]  $ and compact sets
$K_{1},K_{2},...K_{n+1}\subset X$ as follows:

\begin{itemize}
\item let $t_{1}=1$ and choose $K_{1}$ so that $H(Z\times\{1\})\subset K_{1}$

\item for $i=2,3,...,n$, choose $t_{i}$ so that $H(Z\times\lbrack0,t_{i}])\cap
K_{i-1}=\emptyset$ and $K_{i}\subset X$ so that $H(Z\times\lbrack
t_{i},1])\cup K_{i-1}\subset K_{i}$ and $K_{i}$ contains all elements of
$\mathscr{U}^{\prime}$ that intersect $K_{i-1}$. (By properness, elements of
$\mathscr{U}^{\prime}$ have compact closures in $X$.)

\item let $t_{n+1}=0$ and $K_{n+1}=\overline{X}$.
\end{itemize}

Let $\lambda:[0,1]\rightarrow\lbrack0,1]$ be piecewise linear with
$\lambda(0)=0\,,\,\lambda(1)=1\,,\,\hbox{ and }\lambda\left(  \frac{i}%
{n}\right)  =t_{n-i+1}$. Reparametrize $H$ using $\lambda$ and then push
$Z\times\lbrack0,1]$ completely into $X$ by using the map $F:Z\times
\lbrack0,1]\rightarrow X$ defined by $F(z,s)=H(z,\lambda(s))$ for $s\in\left[
\frac{1}{n},1\right]  $ and $F(z,s)=H\left(  z,\frac{1}{n}\right)  $ for
$s\in\left[  0,\frac{1}{n}\right]  $.

We show that $\mathscr{V}=\{F^{-1}(U)|U\in\mathscr{U}^{\prime}\} $ is an open
cover of $Z\times[0,1]$ with mesh at most $\epsilon$ and order at most $n+1$.

Let $(z,s), (z^{\prime},s^{\prime})\in F^{-1}(U)$ and set $y=F(z,s),
y^{\prime}=F(z^{\prime},s^{\prime})$ and $t=\lambda(s), t^{\prime}%
=\lambda(s^{\prime})$. Choose $j\in\{1,2,...,n+1\}$ such that $y\in
K_{j}-K_{j-1}$. By the choice of $K_{i}$ and $t_{i}^{\prime}s$ above,
$t_{j+1}<t< t_{j-1}$. Thus, $\frac{n-j}{n}<s<\frac{2+n-j}{n}$. Since
$y,y^{\prime}\in U$ and $y\in K_{j}$, then $U\cap K_{j}=\emptyset$, so
$y^{\prime}\in K_{j+1}$. Furthermore, $y^{\prime}\notin K_{j-2}$ because if it
were, $U\cap K_{j-2}\neq\emptyset$ and $U\subset K_{j-1}$, a contradiction to
the choice of $j$. Thus, $y^{\prime}\in K_{j+1}-K_{j-2}$. Similar reasoning as
above for $t$ shows that $t_{j+2}<t^{\prime}<t_{j-2}$ and $\frac{n-1-j}%
{n}<s^{\prime}<\frac{n+3-j}{n}$. Thus,
\[
|s-s^{\prime}|<\frac{n+3-j}{n}-\frac{n-j}{n}=\frac{3}{n}<\epsilon
\]

Moreover,

\[
\overline{d}(z,z^{\prime})\leq\overline{d}(z,y)+\overline{d}(y,y^{\prime
})+\overline{d}(y^{\prime},z^{\prime})
\]
\[
=\overline{d}(z,H(z,\lambda(s)))+\overline{d}(y,y^{\prime})+\overline
{d}(z^{\prime}H(z^{\prime},\lambda(s^{\prime})))
\]
\[
<\frac{\epsilon}{3}+\frac{\epsilon}{3}+\frac{\epsilon}{3}=\epsilon
\]

By the above $d_{\infty}((z,s),(z^{\prime},s^{\prime}))<\epsilon$, proving
mesh$_{d_{\infty}}\mathscr{V}<\epsilon$. Since $\mathscr{V}$ consists of the
pre-images of $\mathscr{U}^{\prime}$ and order$\mathscr{U}^{\prime}\leq n+1$,
then order$\mathscr{V}\leq n+1$. Using the definition of dimension in Lemma 4
we have dim$(Z\times[0,1])\leq n$ and an application of Lemma 10 finishes the claim.
\end{proof}

\begin{remark}
\label{Remark: generalized Z-sets}
\emph{We have chosen to follow the traditional definition of} $\Z$-\emph{sets and require} $\overline{X}$ \emph{to be an ANR. However, the above proof also applies to more general metric spaces. In particular, we make no use of the ANR properties of} $\overline{X}$ \emph{or} $X$;  \emph{if }$Z$ \emph{is a closed subset of any compact metric space} $\overline{X}$ \emph{and it is possible to instantly homotope }$\overline{X}$\emph{ off of} $Z$,\emph{ then the proof of Theorem 2 will go through as above.} \end{remark}

From Theorem \ref{Theorem: controlled Z-compactification} we obtain a correct
proof of the main assertion of \cite[Cor.1.2]{ChHo12}, which does not involve groups.

\begin{cor}
If $X$ is a proper CAT(0) space, then $\operatorname*{asdim}X\geq\dim\partial
X+1$.

\end{cor}

To obtain Theorem \ref{Theorem: dimZ+1<=asdimG}, we first must show that the
notion of controlled $\mathscr{Z}$-compactification applies to a
$\mathscr{Z}$-structure $\left(  \overline{X},Z\right)  $ on a group $G$.
Since $Z\subseteq\overline{X}$ is a $\mathscr{Z}$-set, all that remains to
show is that open balls in $X$ become small near the boundary. The cocompact
action by isometries combined with the nullity condition will grant that control.

\begin{lemma}
Suppose a group $G$ admits a $\mathscr{Z}$-structure, $(\overline{X},Z)$. Then
$\overline{X}$ is a controlled $\mathscr{Z}$-compactification of
$X=\overline{X}-Z$.
\end{lemma}

\begin{proof}

Let $\epsilon>0$ and $R>0$. By cocompactness, there is a compact set $C\subset
X$ such that $X\subset GC$. Choose $d>0$ and $x_{0}\in X$ such that $C\subset
B(x_{0},d)$. By the nullity condition, there is a compact set $K^{\prime
}\subset X$ such that whenever $gB(x_{0},d+R)\cap K^{\prime}=\emptyset$ for
some $g\in G$, then $\hbox{diam}_{\overline{d}}gB(x_{0},d+R)<\epsilon$. Let
$K=\overline{N_{2d+R}(K^{\prime})}$ be the closed $2d+R$ neighborhood of
$K^{\prime}$ in $X$. We show this is the desired compact set. Thus, let
$B(x,R)\subset X$ for some $x\in X$ with $B(x,R)\cap K=\emptyset$. Choose
$g\in G$ such that $gx\in C$. Then, $B(x,R)\subset g^{-1}B(x_{0},d+R)$ since
for any $y\in B(x,R)$,
\[
d(y,g^{-1}x_{0})\leq d(y,x)+d(g^{-1}x,x_{0})<R+d
\]

Furthermore, $g^{-1}B(x_{0},d+R)\cap K^{\prime}=\emptyset$. Otherwise, there
would be some \newline$z\in g^{-1}B(x_{0},d+R)\cap K^{\prime}$ and $d(x,z)\leq
d(x,g^{-1}x_{0})+d(g^{-1}x_{0})<2d+R$. However, $B(x,R)\cap K=\emptyset$, so,
$d(x, K^{\prime})>2d+R$. Because $z\in K^{\prime}$, we obtain the required contradiction.

Thus $\hbox{diam}_{\overline{d}}g^{-1}B(x_{0},d+R)<\epsilon$. $B(x,R)$, being
a subset of $g^{-1}B(x_{0},d+R)$, will also have diameter smaller than
$\epsilon$.
\end{proof}

\begin{proof}
[Proof of Theorem \ref{Theorem: dimZ+1<=asdimG}]Suppose a group $G$ admits a
$\mathscr{Z}$-structure $(\overline{X},Z)$. By Lemma 12, $\overline{X}$ is a
controlled $\mathscr{Z}$-compactification of $X$. Thus, by Theorem
\ref{Theorem: controlled Z-compactification}, $\hbox{asdim}X\geq
\hbox{dim}Z+1$. Since $G$ acts geometrically on $X$, $G$ is coarsely
equivalent to $X$ (see Corollary 0.9 in \cite{BDM07}). Moreover, by
\cite{Ro03}, asymptotic dimension is a coarse invariant; so
$\hbox{asdim}X=\hbox{asdim}G$.
\end{proof}

\bibliographystyle{plain}
\bibliography{Biblio}
{}

\end{document}